\documentclass[12pt,twoside,a4paper,reqno]{amsart}

\usepackage{fullpage}

\usepackage{amsmath,amsthm,amsfonts,amssymb}
\usepackage{amsthm}
\usepackage{graphicx}
\usepackage{hyperref}
\usepackage{color}
\usepackage[utf8]{inputenc}
\usepackage{xy}
\usepackage{graphics}
\usepackage{mathrsfs}
\usepackage{comment}
\usepackage{xcolor}

\xyoption{all}

\makeatletter
\newtheorem*{rep@theorem}{\rep@title}
\newcommand{\newreptheorem}[2]{%
\newenvironment{rep#1}[1]{%
 \def\rep@title{#2 \ref{##1}}%
 \begin{rep@theorem}}%
 {\end{rep@theorem}}}
\makeatother

\numberwithin{equation}{section}
\theoremstyle{plain}
\newtheorem{theorem}{Theorem}[section]
\newreptheorem{theorem}{Theorem}
\newtheorem{proposition}[theorem]{Proposition}
\newtheorem{lemma}[theorem]{Lemma}
\newtheorem{corollary}[theorem]{Corollary}
\newreptheorem{corollary}{Corollary}

\theoremstyle{definition}
\newtheorem{definition}[theorem]{Definition}
\newtheorem{remark}[theorem]{Remark}

\usepackage{tikz-cd}

\begin{document}
\title{On the Rationality of moduli spaces of vector bundles over chain-like curves}
\author{Suhas B. N., Praveen Kumar Roy and Amit Kumar Singh}
\address{Department of Mathematics, St. Joseph’s College (Autonomous), Bangalore, India}
\email{chuchabn@gmail.com / suhas.b.n@sjc.ac.in}
\address{School of Mathematics, Tata Institute of Fundamental Research, Mumbai 400005, India }
\email{praveen.roy1991@gmail.com}
\address{Chennai Mathematical Institute, H1 SIPCOT IT Park, Siruseri, Kelambakkam 603103, India}
\email{amitksingh@cmi.ac.in/ amitks.math@gmail.com}
\maketitle

\begin{abstract} Let $C$ be a chain-like curve over $\mathbb{C}$. In this paper, we investigate the rationality of moduli spaces of $w$-semistable vector bundles on $C$ of arbitrary rank and fixed determinant by putting some restrictions on the Euler characteristics. 
\end{abstract}

\section{Introduction}
The moduli space of vector bundles (of fixed rank and degree) on curves is one of the most important and well explored topics in algebraic geometry. Over a smooth projective curve $C$ of genus $g$, it is not possible to parametrize the set of all isomorphism classes of vector bundles of a fixed rank and degree by means of an algebraic variety or a scheme. To address this issue, David Mumford introduced the concept of stable and semistable vector bundles and proved that over a smooth projective curve $C$ of genus $g\geq 2$, the isomorphism classes of stable bundles have the structure of a smooth quasi-projective variety \cite{Mumford-GIT-1965}. A natural compactification of this variety was then given by Seshadri by first introducing the S-equivalence relation on semistable vector bundles and then including the S-equivalence classes of semistable bundles in the moduli space. This compactified moduli space of semistable bundles of rank $r$ and degree $d$, denoted by $\mathcal{U}_C(r,d)$, turns out to be a normal irreducible projective variety of dimension $r^2(g-1)+1$. When $(r,d)=1$, the notion of semistability coincides with that of stability, and so in this case, $\mathcal{U}_C(r,d)$ becomes a smooth projective variety. One also has a canonical determinant map $\det : \mathcal{U}_C(r,d) \rightarrow \mathcal{U}_C(1,d)$, which sends an equivalence class $[E] \in \mathcal{U}_C(r,d)$ to its determinant class $[\Lambda^rE] \in \mathcal{U}_C(1,d)$. The fibre under this map of any arbitrary element $L \in \mathcal{U}_C(1,d)$ is denoted by $\mathcal{SU}_C(r,L)$, and is called a fixed determinant moduli space.
 
 The question of rationality of the fixed determinant moduli space of stable vector bundles over smooth projective curves have been quite well explored. In 1966, Tjurin  first proved that the moduli space $\mathcal{SU}_{C}(2, L)$ is a rational variety, for any line bundle $L$ on $C$ of odd degree \cite[Theorem 14]{Tjurin-1966}. Subsequently in 1975, Newstead generalised the result of Tjurin and proved that $\mathcal{SU}_{C}(r, L)$ is a rational variety provided $(r, d)= 1$ and some other conditions hold. His technique was to construct a family of vector bundles of the required type, parametrized by a Zaraski open subset $S$ of an affine space such that the induced map from $S$ to $\mathcal{U}_C(r, L)$ is birational  \cite{Newstead-1975,Newstead-1980}. After a couple of decades, it was proved by King and Schofield in \cite{King-Schofield-1999} that $\mathcal{U}_C(r , d)$ is birationally equivalent to $\mathcal{U}_C(h, 0) \times \mathbb{A}^{(r^2 - h^2)(g-1)}$, where $h=(r, d)$.  Further, they also proved that $\mathcal{SU}_C(r , L)$ is a rational variety when $(r, d)=1$, thereby settling the rationality question completely for the case when rank and degree are coprime. Nevertheless, the question is still mostly open for the non-coprime case. Narasimhan and Ramanan have proved the rationality of $\mathcal{SU}_C(2, L)$ when the curve $C$ is of genus 2 and the line bundle $L$ is of even degree \cite{Narasimhan-Ramanan-1969}.
 
 In the case when the base curve $C$ is an irreducible nodal curve of arithmetic genus $p_a(C) \geq 2$, one has to include torsion-free sheaves and not just vector bundles in order to get a compact moduli space (cf. \cite[Chapter VII]{Seshadri-1982}, \cite[Chapter V]{Newstead-2011}). Let $\mathcal{U}_C(r,d)$ denote the moduli space of semistable torsion-free sheaves on $C$ of rank $r$ and degree $d$. Then, motivated by \cite{King-Schofield-1999}, Bhosle and Biswas proved that $\mathcal{U}_C(r, d)$ is birational to $\mathcal{U}_C(h, 0) \times \mathbb{A}^{(r^2 - h^2)(p_a(C)-1)}$, and that $\mathcal{SU}_C(r, L)$ is rational if $h=1$ \cite[Theorem 1.2]{Bhosle-Biswas-2014}. Further, they also proved that $\mathcal{SU}_C(r, L)$ is rational when $p_a(C) = 2 = r$ \cite[Theorem 1.2]{Bhosle-Biswas-2014}.
 
 Suppose $C$ is a reducible nodal curve with $n$ smooth components $C_i$ of genus $g_i \geq 2$ and $n-1$ nodes $p_i$ such that $C_i \cap C_j = \emptyset$, whenever $|i-j| >1$ and $C_i \cap C_{i+1} = \{p_i\}$, for $i=1,\dots,n-1$. Such curves are called \textit{chain-like curves}. In this case, to get a compact moduli space, one replaces the vector bundles by pure sheaves of dimension one. Also one needs to consider weights of the components (called polarization in the literature) in order to define semistability. Let $\mathcal{U}_C(w,r,\chi)$ denote the moduli space of $w$-semistable pure sheaves of dimension one and Euler characteristic $\chi$. In \cite{Bigas91}, Teixidor identifies the irreducible components of $\mathcal{U}_C(w,r,\chi)$ (for further details, see section 2). Now suppose $L$ is a fixed line bundle on $C$ of suitable Euler characteristic. Let $\overline{\mathcal{SU}_C(w,r,\chi,L)}$ denote the closure of the collection of all vector bundles in $\mathcal{U}_C(w,r,\chi)$ with determinant $L$. For the case when $n=2=r$, it was proved in \cite{Barik-Dey-Suhas-2018} by Barik, Dey and Suhas that $\overline{\mathcal{SU}_C(w,2,\chi,L)}$ is rational whenever $(2,\chi)=1$. More recently, Favale and Brivio have proved it for higher rank under some extra assumption on the Euler characteristic \cite{Brivio-Favale-VB-2020}. For $n \geq 2$ and $r=2$, it has been shown in \cite{Dey-Suhas-2018} that $\overline{\mathcal{SU}_C(w,2,\chi,L)}$ is rational provided $(2,\chi)=1$.
 
 In this article, we investigate the rationality of $\overline{\mathcal{SU}_C(w,r,\chi,L)}$ for $r \geq 2$ and $n \geq 2$ under some conditions on the Euler characteristic. Though we use the technique employed in \cite{Brivio-Favale-VB-2020}, increase in the number of components of the base curve and hence the complexity involved, makes it more challenging. The main result of this article is the following theorem.
 \begin{reptheorem}{theorem on rationality}
  Let $C$ be a chain-like curve as mentioned above. Fix $r \geq 2$, and $d_i \in \mathbb{Z}$ with $(r, d_i) = 1$, for $i = 1,\dots, n$. Let $\chi_i = d_i + r (1 -g_i)$ and $\chi = \sum\limits_{i=1}^{n} \chi_i - (n-1)r$.  For any line bundle $L$ of multidegree $(d_1,\dots,d_n)$ and any polarization $w = (w_1, \dots, w_n)$ satisfying the conditions
\[
(\sum\limits_{i=1}^j w_i)\chi - \sum\limits_{i=1}^{j-1}\chi_i + r(j-1) < \chi_j < (\sum\limits_{i=1}^j w_i)\chi - \sum\limits_{i=1}^{j-1}\chi_i + rj,
\]
for $1\leq j \leq n-1$, $\overline{\mathcal{SU}_C(w,r,\chi,L)}$ is a rational variety.
\end{reptheorem}

The coprimality condition in the above theorem is very strong. For example, if $r=2$, then there is at most one 
choice for the $\chi_j$ satisfying the coprimality assumptions and the inequalities (exactly one for generic $w$). This means 
that Theorem \ref{theorem on rationality} covers only one of the $2^{n-1}$ components of the moduli space. 
Note that this result is very different from 
that of \cite{Dey-Suhas-2018}, where the authors prove the rationality of all the components when $(2,\chi) =1$. However, 
when $n$ and $\chi$ are both even, \cite{Dey-Suhas-2018} gives no information, whereas there is a component with $(2,d_i) =1$ 
for all $i$.

The proof of Theorem \ref{theorem on rationality} involves the construction of a birational model for $\overline{\mathcal{SU}_C(w,r,\chi,L)}$ and indeed for $\mathcal{U}_{C}(w,r, \chi)_{d_1, \dots , d_n}$ 
(Theorem \ref{theorem for the first birational map} and Corollary \ref{main corollory}). In fact, this model can be constructed 
without the coprimality assumptions (see Remark \ref{birational model without coprimality conditions}).

 \section{Preliminaries}\label{section on Chain-like curves}
 Let $C$ be a reduced nodal curve having $n$ smooth irreducible components $C_i$ of genus $g_i$ with $C_i \cap C_{i+1} = \lbrace p_i \rbrace$ and $C_i \cap C_j = \emptyset$ whenever $\vert i - j \vert >1$. Such a curve $C$ is called a chain-like curve. Let $\nu : \tilde{C} \rightarrow C$ be the normalisation of $C$ such that  $\nu^{-1} (p_i) = \lbrace q_i,q_i' \rbrace $, where $q_i \in C_i$ and $q_i' \in C_{i + 1}$.
On such a curve $C$, we have the following exact sequence:
\begin{equation}\label{canonical exact sequence}
    0 \rightarrow \mathcal{O}_C \rightarrow \bigoplus_{j=1}^n\mathcal{O}_{C_j} \rightarrow \mathcal{T} \rightarrow 0.
\end{equation}
Here $\mathcal{T}$ is supported only at the nodal point(s).
 Also, it is not hard to show that each $\mathcal{T}_{p_i}$ is a one dimensional vector space over $\mathbb{C}$ and that $h^0(\mathcal{T}) = n-1$.
 From the exact sequence \eqref{canonical exact sequence}, we have 
\begin{eqnarray}
    \chi(\mathcal{O}_C) & = & \sum\limits_{j=1}^n \chi(\mathcal{O}_{C_j}) - \chi(\mathcal{T}) \nonumber \cr
    & = & n - \sum\limits_{j=1}^n g_j - (n-1)
    \nonumber \cr
    & = & 1 - \sum\limits_{j=1}^n g_j.
\end{eqnarray}
So, if we let $p_a(C) = 1 - \chi(\mathcal{O}_C)$ be the arithmetic genus of $C$, then $p_a(C) = \sum\limits_{j=1}^n g_j$.
Now, suppose $E$ is a vector bundle of rank $r$ on $C$. Tensoring the exact sequence (\ref{canonical exact sequence}) with $E$,  denoting $E_{|_{C_j}}$ by $E_j$ and the cokernel by $\mathcal{T}_E$, we obtain 
\begin{equation}\label{canonical sequence for the vector bundle E}
  0 \rightarrow E \rightarrow \bigoplus_{j=1}^n E_j \rightarrow \mathcal{T}_E \rightarrow 0. 
\end{equation}
We also have 
\begin{align}\label{equation on computing euler seq of vb on chain-like curve}
\chi(E) = \sum\limits_{j=1}^n \chi(E_j) - r(n-1).
\end{align}

\begin{definition}
 Let $F$ be a coherent sheaf of $\mathcal{O}_C$-modules. We call $F$ a pure sheaf of dimension one if; for every nonzero proper $\mathcal{O}_C$-subsheaf  $G \subset F$, the dimension of the support of $G$ is equal to one.
\end{definition}
Vector bundles on $C$ are examples of pure sheaves of dimension one. Suppose $F$ is a pure sheaf of dimension one on $C$. Let $F_j = \frac{F_{|_{C_j}}}{\text{Torsion}(F_{|_{C_j}})}$ for each $j$, where $\text{Torsion}(F_{|_{C_j}})$ is the torsion subsheaf of $F_{|_{C_j}}$. Then $F_j$, if non-zero, is torsion-free and hence locally free on $C_j$. Let $r_j$ denote the rank of $F_j$ and let $d_j$ denote the degree of $F_j$ for each $j$. We call the n-tuples $(r_1,\dots,r_n)$ and $(d_1,\dots,d_n)$ respectively, the \textit{multirank} and the \textit{multidegree} of $F$.
\begin{definition}\label{polarization}
 Let $w = (w_1,\dots,w_n)$ be an n-tuple of rational numbers such that $0 < w_j < 1$ for each $j$ and $\sum\limits_{j=1}^n w_j =1$. We call such an n-tuple a polarization on $C$.
\end{definition}

We now define the slope of a pure sheaf of dimension one on a chain-like curve $C$ with respect to a given polarization $w$.

\begin{definition}\label{polarized slope}
Let $F$ be a pure sheaf of dimension one on $C$ of multirank $(r_1,\dots,r_n)$. Then the slope of $F$ with respect to a polarization $w$, denoted by $\mu_w(F)$, is defined by $\mu_w(F) = \frac{\chi(F)}{\sum\limits_{j=1}^n w_jr_j}$.
\end{definition}

\begin{definition}
 Let $E$ be a pure sheaf of dimension one on $C$. Then, $E$ is said to be $w$-\textit{semistable} (resp. $w$-\textit{stable}) if for any proper subsheaf $F \subset E$, one has $\mu_w(F) \leq \mu_w(E)$ (resp. $\mu_w(F) < \mu_w(E)$). If $E$ is not $w$-semistable for any polarization $w$ on $C$, we call it a \textit{strongly unstable}.
\end{definition}

Now, let $E$ be a $w$-semistable pure sheaf of dimension one on $C$. Then there is a finite filtration of pure sheaves of dimension one on $C$
\begin{align}\label{filtration of Jordan Holder filtration}
0 = E^0 \subset E^1 \subset \cdots \subset E^{k-1} \subset E^{k} = E,
\end{align}
such that, for each $1\leq j\leq k$, the quotient $E^j / E^{j-1}$ is a $w$-stable pure sheaf of dimension one on $C$ with $\mu_w (E^j / E^{j-1}) = \mu_w (E)$. Such filtration is known as a {\em Jordan-H\"older filtration} of $E$. Given a Jordan  H\"older filtration \eqref{filtration of Jordan Holder filtration}, define the {\em graded sheaf} associated to $E$ as
\begin{align}
Gr_w(E) = \bigoplus_{j =1}^{k} E^j / E^{j-1},
\end{align}
which depends on the isomorphism class of $E$. Two $w$-semistable pure sheaves $E$ and $F$, each of dimension one on $C$ are said to be $S_w$-equivalent if 
\begin{align*}
Gr_w(E) = Gr_w(F).
\end{align*} 
If $E$ and $F$ are $w$-stable sheaves, then $S_w$-equivalence is equivalent to an isomorphism. \\
The following result from \cite[Theorem-1, Steps-1,2]{Bigas91} will be useful in subsequent theorems. 
\begin{theorem}\label{theorem of Bigas}
 Suppose $C$ is a chain-like curve and $E$ is a vector bundle on $C$ of rank $r$. Let $E_i$ denote the restriction of $E$ to the component $C_i$ for $i = 1,\dots,n$ and $w$ denote a polarization on $C$. Then 
 \begin{enumerate}
     \item[(1)] if $E$ is $w$-semistable, $\chi$ is the Euler characteristic of $E$ and $\chi_j$ the Euler characteristic of $E_j$ for each $j$, then
     \begin{equation}\label{polarization inequalities}
         (\sum\limits_{i=1}^j w_i)\chi - \sum\limits_{i=1}^{j-1}\chi_i + r(j-1) \leq \chi_j \leq (\sum\limits_{i=1}^j w_i)\chi - \sum\limits_{i=1}^{j-1}\chi_i + rj,
     \end{equation}
     where $j$ varies from $1$ to $n-1$;
     \item[(2)] if $E_i$ is a semistable vector bundle for $i = 1,\dots,n$ and the Euler characteristics of $E$ and $E_i$ satisfy the inequalities \eqref{polarization inequalities}, then $E$ is $w$-semistable. Further, if the inequalities \eqref{polarization inequalities}  are strict for $i=1,\dots, n-1$ and at least one $E_i$ is stable, then $E$ is $w$-stable.
     \item[(3)] The irreducible components of $\mathcal{U}_C(w,r,\chi)$ are indexed by the multidegrees $(d_1,\dots,d_n)$ for 
     which $\chi_i = d_i -r(g_i-1)$ for $1\leq i \leq n$, such that the $\chi_i$ satisfy the strict inequalities \eqref{polarization inequalities} for $1\leq i \leq n-1$ and $\sum_{i}^n\chi_i -r(n-1) = \chi$.
 \end{enumerate}
\end{theorem}
From now on, we denote the irreducible components of $\mathcal{U}_C(w,r,\chi)$ by $\mathcal{U}_C(w,r,\chi)_{d_1,\dots,d_n}$.

 \section{Construction of pure sheaves of dimension one}

 Let $C$ be a smooth projective curve over the base field $\mathbb{C}$ with genus $g \geq 2$. Let $\mathcal{U}_C(r,d)$ be the moduli space of semistable vector bundles of rank $r$ and degree $d$ on a curve $C$. It is a quasi-projective variety and smooth projective when $r$ and $d$ are co-prime \cite{Seshadri-1982}. Moreover, 
\[ 
 \dim \mathcal{U}_C(r,d)= r^2(g-1) +1
 \] (see \cite{Seshadri-1982}). Note that in the co-prime case, the notions of stability and semistability coincide. 
 
Let $C$ be a chain-like curve having $n$ smooth components $C_i$ with $(n-1)$ nodes $p_i$. Let $r$ and $d_1, \dots , d_n$ be integers with $r \geq 1$.
In this section, we construct a fibre product of projective bundles $ \mathbb{P}$ over $\mathcal{U}_{C_1}(r, d_1) \times \cdots \times \mathcal{U}_{C_n}(r, d_n)$ and then we also prove the existence of a polarization $w$ (see \ref{Existence of polarization}) such that the induced map from $ \mathbb{P}$  to $\mathcal{U}_C(w,r, \chi)_{d_1,\dots , d_n}$ is birational.
Here, $\mathcal{U}_C(w,r, \chi)_{d_1,\dots , d_n}$ denotes the irreducible component of the moduli space of $w$-semistable  vector bundles of rank $r$ and multidegree $(d_1,\dots , d_n)$ on a chain-like curve $C$ with Euler characteristic $\chi$.

\begin{theorem} \label{theorem on exists on poincare bundle on chain-like curve}
Let $C$ be a chain-like curve having $n$ smooth irreducible components $C_i$ of genus $g_i \geq 1$ with nodes $p_i$ and  $\nu^{-1} (p_i) = \lbrace q_i,q_i' \rbrace $. Fix $r \geq 2$, and $d_1,\dots,d_n \in \mathbb{Z}$ such that $r$ is pairwise co-prime with  $d_1,\dots,d_n$. Then there is a locally trivial map
\[
\pi : \mathbb{P} \rightarrow \mathcal{U}_{C_1}(r, d_1) \times \cdots \times \mathcal{U}_{C_n}(r, d_n),
\]
where $\mathbb{P}$ is a fibre product of projective bundles, such that the fibre at $([E_1],\dots,[E_n])$ is 
\[
 ( \mathbb{P}(\text{Hom}(E_{1,q_1}, E_{2, q_1'})) , \dots , \mathbb{P}( \text{Hom}(E_{n-1,q_{n-1}}, E_{n, q_{n-1}'}))).
 \]
Here $E_{i, q_i}$ (resp. $E_{i+1, q_{i}'}$) denotes the fibre of $E_i$ (resp. $E_{i+1}$) at $q_i$ (resp. $q_{i}'$).
\end{theorem} 
\begin{proof}
The proof follows the constructions of Lemma 3.1 of \cite{Brivio-Favale-VB-2020}, generalizing it to chain-like curves.

Since $r$ is pairwise co-prime with $d_1,\dots,d_n$, for each $i$ there exists a Poincar\'e bundle $\mathcal{P}_i$ for the moduli space of semistable vector bundles on $C_i$ of rank $r$ and degree $d_i$. In other words, for each $i$, we have the Poincar\'e bundle $\mathcal{P}_i$ on $\mathcal{U}_{C_i}(r, d_i) \times C_i$ with $\mathcal{P}_i \vert_{[E_{i}] \times C_i} \cong E_i$ via the identification $[E_i] \times C_i \cong C_i$.  Let 
\begin{align*}
\iota_i : \mathcal{U}_{C_i}(r, d_i) \times q_i \hookrightarrow \mathcal{U}_{C_i}(r, d_i) \times C_i, 
\end{align*}
\begin{align*}
\iota_i' : \mathcal{U}_{C_{i+1}}(r, d_{i+1})  \times q_{i}' \hookrightarrow \mathcal{U}_{C_{i+1}}(r, d_{i+1}) \times C_{i+1}, 
\end{align*}
be the natural inclusions, and $\iota_i^*(\mathcal{P}_i)$, ${\iota'}_i^*(\mathcal{P}_{i+1})$ the pullbacks under these inclusions of the Poincar\'e bundle $\mathcal{P}_i$, where $i = 1,\dots, n-1$. Note that  $\mathcal{U}_{C_i}(r, d_i) \times q_i \cong \mathcal{U}_{C_i}(r, d_i)$, and $\mathcal{U}_{C_{i+1}}(r, d_{i+1}) \times q_{i}' \cong \mathcal{U}_{C_{i+1}}(r, d_{i+1})$,  (1 $\leq i \leq n-1)$.  Therefore, the vector bundle  $\iota_i^*(\mathcal{P}_i)$ (resp. ${\iota'}_i^*(\mathcal{P}_{i+1})$) can  be considered as a vector bundle on $\mathcal{U}_{C_i}(r, d_i)$ (resp. $\mathcal{U}_{C_{i+1}}(r, d_{i+1})$) of rank $r$ for $i = 1,\dots, n-1$. Moreover, the fibre of $\iota_i^*(\mathcal{P}_i)$ at $[E_i]$ is $E_{i, q_i}$, and the fibre of ${\iota'}_i^*(\mathcal{P}_{i+1})$ at $[E_{i+1}]$ is $E_{{i+1}, q_{i}'}$. \\
The product $\mathcal{U}_{C_1}(r, d_1) \times \cdots \times \mathcal{U}_{C_n}(r, d_n)$ is a smooth irreducible projective variety. Let $\pi_i :  \mathcal{U}_{C_1}(r, d_1) \times \cdots \times \mathcal{U}_{C_n}(r, d_n) \rightarrow \mathcal{U}_{C_i}(r, d_i)$ be the $i^\text{th}$ projection map. We define the following sheaves on $\mathcal{U}_{C_1}(r, d_1) \times \cdots \times \mathcal{U}_{C_n}(r, d_n)$:
\begin{align}
\mathcal{F}_i := \mathcal{H}om(\pi^*_i(\iota^*_i(\mathcal{P}_i)), \pi^*_{i+1}({\iota'}^*_i (\mathcal{P}_{i+1}))).
\end{align}
Let 
\[
\pi: \mathbb{P} \rightarrow \mathcal{U}_{C_1}(r, d_1) \times \cdots \times \mathcal{U}_{C_n}(r, d_n)
\]
denote the fibre product of $\mathbb{P}(\mathcal{F}_i)$ over $\mathcal{U}_{C_1}(r, d_1) \times \cdots \times \mathcal{U}_{C_n}(r, d_n)$.
By construction, the fibre of $\pi$ at a point $([E_1],\dots,[E_n])$ is $\mathbb{P}(\text{Hom}(E_{1,q_1}, E_{2, q_1'})) \times \dots \times \mathbb{P}(\text{Hom}(E_{n-1,q_{n-1}}, E_{n, q_{n-1}'}))$.
\end{proof}

\begin{remark}\label{birational model without coprimality conditions}
Note that $\mathbb{P}$ can be defined without the coprimality conditions in the above theorem. 
Let $\mathcal{Q}_i = \mathcal{Q}(\mathcal{O}_{C_i}^{N_i}/P_i)$ denote the Grothendieck's Quot scheme parameterising 
the quotients of $\mathcal{O}_{C_i}^{N_i}$ having fixed Hilbert polynomial $P_i$. Then, for a suitable choice of $N_i$ and $P_i$, there is an open subset $\mathcal{R}_i \subset \mathcal{Q}_i$ such that the action of $\text{GL}(N_i)$ on it descends to an action of $\text{PGL}(N_i)$, and the 
quotient is $\mathcal{U}_{C_i}^s(r,d_i)$ (the moduli space of stable bundles of rank $r$ and degree $d_i$ on $C_i$). Also, there exists a projective Poincar\'e bundle $\mathcal{P}_i$ on 
$\mathcal{R}_i\times C_i$, whose restriction on $\{E\} \times C_i$ is isomorphic to $\mathbb{P}(E)$ for all $E \in \mathcal{R}_i$ 
(see \cite{Biswas-Paz-Newstead} for details). 
Now, we can construct the bundles $\mathcal{F}_i$ on $\mathcal{R}_1\times \cdots \times \mathcal{R}_n$ and the action of 
$\text{GL}(N_1) \times \cdots \times \text{GL}(N_n)$ on $\mathbb{P}(\mathcal{F}_i)$ descends to an action of 
$\text{PGL}(N_1) \times \cdots \times \text{PGL}(N_n)$. By taking the fibre product of $\mathbb{P}(\mathcal{F}_i)$,
we obtain $\mathbb{P}$ as a bundle on $\mathcal{U}_{C_1}^s(r, d_1) \times \cdots \times \mathcal{U}_{C_n}^s(r, d_n)$, 
with the same fibre as in Theorem \ref{theorem on exists on poincare bundle on chain-like curve}. It is to be noted that this 
bundle is no longer locally trivial in the Zariski topology. However, it does provide a birational model for 
$\mathcal{U}_C(w,r,\chi)_{d_1,...,d_n}$.
\end{remark}

Throughout this paper, by $\mathbb{P}$ we always mean the fibre product of $\mathbb{P}(\mathcal{F}_i)$ as mentioned in 
the previous theorem and remark.

Let $C$ be a chain-like curve. 
Let $j_{p_i}$ be the inclusion of $\lbrace p_i \rbrace$  in $C$, and let  $j_i : C_i \hookrightarrow C$ be the inclusion map for $i = 1,\dots,n$. Then for $1 \leq i \leq n-1 $,  we have the following natural surjective morphisms:
\[
\rho_i : j_{i*}E_i \rightarrow E_{i, q_i} \quad (1 \leq i \leq n-1),
\]
and
\[
{\rho'}_l: j_{l*}E_l \rightarrow E_{l, q_{l-1}'} \quad (2 \leq l \leq n).
\]
The sheaf $j _{p_{i}*}(j_{p_{i}}^* (j_{i+1 *}E_{i+1}))$ is a skyscraper sheaf supported only at $p_i$, and is such that, on any open set containing $p_i$ it is equal to $E_{i+1, q'_i}$. Therefore, we have a canonical isomorphism
\[
\rho_{p_i}' : j _{p_{i}*}(j_{p_{i}}^* (j_{i+1 *}E_{i+1})) \rightarrow E_{i+1, q_{i}'}, 
\]
for each $i$ from 1 to $ n-1$.
Now, let
\begin{align*}
\sigma_i : E_{i,q_i} \rightarrow E_{i+1, q_i'}
\end{align*}
be a linear map for $i = 1,\dots,n-1$. We then have  the induced map
 \[
\sigma_i \oplus \text{id} :E_{i, q_i} \oplus E_{i+1, q_i'} \rightarrow \text{Im}(\sigma_i) \oplus E_{i+1, q_i'}.
\]
Further, let 
\[
\delta_i : \text{Im}(\sigma_i) \oplus E_{i+1, q_i'} \rightarrow E_{i+1, q_i'}
\]
be the map defined by
\[
(u, v) \mapsto u-v,
\]
for $i = 1,\dots, n-1$. We denote the diagonal inside $ \text{Im}(\sigma_i) \oplus \text{Im}(\sigma_i)$ by $\Delta_i$, and $\oplus_{i=1}^{n-1} \Delta_i$ by $\Delta$. Note that $\Delta_i \cong \mathbb{C}^{k_i}$, where $k_i$ denotes the rank of $\sigma_i$.
We then have maps
\[
\tilde{\sigma}_i : j_{i *}(E_{i}) \oplus j_{(i+1)*}(E_{i+1})  \rightarrow  j_{p_i*}(j^*_{p_i} (j_{(i+1) *}(E_{i+1})))
\]
defined by
\[
\tilde{\sigma}_i(u_i,u_{i+1}) = {\rho_{p_i}'}^{-1}(\sigma_i \circ \rho_i(u_i) -\rho_{i+1}'(u_{i+1}))
\]
for $1\leq i \leq n-1$.
Now define 
\[
\tilde{\sigma} : \oplus_{i=1}^n j_{i *}(E_{i}) \rightarrow  \oplus_{i=1}^{n-1}j_{p_i*}(j^*_{p_i} (j_{(i+1) *}(E_{i+1})))
\]
by 
\[
\tilde{\sigma}(u_1,\dots,u_n) := (\tilde{\sigma}_1(u_1,u_2),\dots, \tilde{\sigma}_{n-1}(u_{n-1},u_n)).
\]
Let $K_i := \text{Ker}(\sigma_i \circ \rho_i : j_{i*} E_i \rightarrow \text{Im}(\sigma_i))$ and $K_{i+1}' := \text{Ker}( \rho_{i+1}' : j_{{i+1}*}E_{i+1} \rightarrow E_{i+1, q_i'})$. Then for $2 \leq i \leq n-1$, define 
\[
\tilde{K_i} : = K_i \cap {K_i}',
\]
where $\tilde{K_1} = K_1$ and $\tilde{K_n} = K_n'$. 
This gives the following short exact sequence:
\[
0 \rightarrow \tilde{K_i} \rightarrow {j_i}_*(E_i) \xrightarrow{(\rho_i', \sigma_i\circ \rho_i)} E_{i,q_{i-1}'} \oplus \text{Im}(\sigma_i) \rightarrow 0,
\]
for $2 \leq i \leq n-1$. We then have the following diagram:

\[
    \xymatrix@=.6cm{
     && 0 \ar[dd] && 0 \ar[dd]    
     \\ \\
        && {\rm Ker}(\tilde{\sigma}) \ar[rr]\ar[dd] & & \Delta  \ar[dd]  \\
               \\
             && \oplus_{i=1}^{n} j_{i *}(E_{i}) \ar[rr] \ar[dd]_{\tilde{\sigma}} \ar[rd]^{\oplus^{n-1}_{i=1} (\rho_i, \rho_{i+1}')} & & \oplus_{i=1}^{n-1}(\text{Im}(\sigma_i) \oplus E_{i+1, q_i'}) \ar[dd]_{\oplus_{i=1}^{n-1}\delta_i} 
       \\
        & & & \oplus_{i=1}^{n-1}(E_{i, q_i} \oplus E_{i+1, q_i'}) \ar[ru]^{\oplus^{n-1}_{i=1}({\sigma_i \oplus {\rm id}})} & & & \\
      &&  \oplus_{i=1}^{n-1} j_{p_i*}(j^*_{p_i} (j_{i+1 *}(E_{i+1}))) \ar[rr]^{\oplus_{i=1}^{n-1}\rho_{p_i}'} \ar[dd] && \oplus_{i=1}^{n-1}(E_{i+1, q_i'}) \ar[dd]     
      \\ \\
      && 0  && 0.           
                       }
\]
\begin{definition}\label{def of E_u}
Let $u = ([E_1], \dots , [E_n]; [\sigma_1], \dots , [\sigma_{n-1}]) \in  \mathbb{P}$. We denote the kernel of $\tilde{\sigma}$ by  $E_u$.
\end{definition}
Note that the construction of $E_u$ is similar to those of GPB's which give results for the case of irreducible nodal curves (see \cite{Bhosle-92}).
\begin{lemma}\label{Characterisation of vector bundle}
Let $E_u$ be the sheaf defined by $u = ([E_1],\dots,[E_n]; [\sigma_1],\dots,$ $[\sigma_{n-1}]) \in  \mathbb{P}$. Then, $E_u$ is a pure sheaf of dimension one on a chain-like curve $C$ with $\chi(E_u) = \sum\limits_{i=1}^n \chi_i - (n-1)r$ and multirank $(r,\dots,r)$. Moreover, $E_u$ is a vector bundle if and only if $\sigma_i$ is an isomorphism for $i= 1,\dots,n-1$.
\end{lemma} 
\begin{proof}
Let $k_i$ denote the rank of $\sigma_i$, $1 \leq i \leq n-1$. Note that $E_u$ is equal to $\text{Ker}(\tilde{\sigma})$.
Therefore, $E_u \vert_{p_i} \cong \mathcal{O}_{p_i}^{k_i} \oplus \mathcal{O}_{q_i}^{r-k_i} \oplus \mathcal{O}_{q_i'}^{r-k_i}$, which implies that $E_u$ is a pure sheaf of dimension one \cite{Seshadri-1982}. Moreover, $E_u$ is a vector bundle if and only if $k_i = r$ for all $i= 1 ,\dots,n-1$ which in turn happens if and only if $\sigma_i$ is isomorphism for all $i = 1,\dots, n-1$.
\end{proof} 

\begin{theorem}\label{Section-5}
Let $E = E_u$ be the sheaf defined by  $u = ([E_1],\dots,[E_n]; [\sigma_1],\dots, [\sigma_{n-1}]) \in  \mathbb{P}$. Let $\chi_i$ denote the Euler characteristic of $E_i$ and 
$w=(w_1, \dots, w_n)$ be a polarization. Suppose that the inequalities 
\[
(\sum\limits_{i=1}^j w_i)\chi - \sum\limits_{i=1}^{j-1}\chi_i + r(j-1) < \chi_j < (\sum\limits_{i=1}^j w_i)\chi - \sum\limits_{i=1}^{j-1}\chi_i + rj,
\]
hold for $1\leq j \leq n-1$, and that all $\sigma_i$ are isomorphisms. Then $E$ is a $w$-stable vector bundle.
\end{theorem}
\begin{proof}
 Since $\sigma_i$ are isomorphisms, $E$ is a vector bundle (by Lemma \ref{Characterisation of vector bundle}).  So by Theorem \ref{theorem of Bigas}(2), the result follows.
\end{proof}

The next result is not needed for the proof of our main theorems, but it does give some information about those $E_u$ which are not vector bundles.

\begin{theorem}\label{lemma regarding the existance of the bigas inequality}
Let $E = E_u$ be the sheaf defined by  $
u = ([E_1],\dots,[E_n]; [\sigma_1],\dots, [\sigma_{n-1}]) \in  \mathbb{P}$ with $\text{rank}(\sigma_i) = k_i$. Let $E$ be $w$-semistable with respect to a polarization $w= (w_1, \dots, w_n)$. Then the following conditions are satisfied
\begin{align}\label{eqn:1}
(\sum\limits_{i =1}^{j} w_i)  \chi - \sum\limits_{i=1}^{j-1} \chi_i -\sum\limits_{i = j+1}^{n-1} k_i + (j-1)r \leq \chi_j \leq (\sum\limits_{i =1}^{j} w_i)  \chi - \sum\limits_{i=1}^{j-1} \chi_i + \sum\limits_{i=1}^{j} k_i + (j-1)r,
\end{align}
where $j = 1,\dots,n-1$. Moreover, if $E$ is $w$-stable, then the inequalities hold strictly.
\begin{proof}
Since $E$ is $w$-semistable and $\tilde{K_i}$ are subsheaves of $E$, we have: 
\[
\mu_w (\tilde{K_i}) \leq \mu_w (E) \quad (1\leq i \leq n).
\]

This gives
\begin{align}\label{inquality  the first when w stable vector bundle}
\chi_1  \leq  w_1 \chi + k_1  \cr
\chi_n  \leq  w_n \chi + r.
\end{align}
and 
\begin{align}\label{inquality  the second when w stable vector bundle}
\chi_l \leq  w_l \chi + r + k_l\quad (2 \leq l \leq n-1).
\end{align}

For $j \in \lbrace 1 , \dots , n-1 \rbrace$, we rewrite \eqref{equation on computing euler seq of vb on chain-like curve},
\begin{align}\label{equation computing Euler characteristics of chain-like curve staring from L}
\chi_j = 
\chi - \sum_{\substack{i = 1,\\ i \neq j}}^{n} \chi_i + (n-1)r. 
\end{align}
From \eqref{inquality  the first when w stable vector bundle} and \eqref{inquality  the second when w stable vector bundle}, we have
\begin{align}
\sum\limits_{i = j+1}^n \chi_i \leq (\sum\limits_{i= j+1}^n w_i) \chi + \sum\limits_{i = j+1}^{n-1}k_i + (n-j)r.
\end{align}
Substituting the values of $\sum\limits_{i = j+1}^n \chi_i$ in \eqref{equation computing Euler characteristics of chain-like curve staring from L} we get,
\begin{equation*} \label{eq1}
\begin{split}
 \chi_j & \geq \chi - \sum\limits_{i =1}^{j-1} \chi_i - (\sum\limits_{i= j+1}^n w_i) \chi - \sum\limits_{i = j+1}^{n-1}k_i +(j-1)r  \\
 & = (1 - \sum\limits_{i = j +1}^n w_i) \chi - \sum\limits_{i =1}^{j-1} \chi_i - \sum\limits_{i = j +1}^{n-1} k_i + (j-1)r  \\
 & = (\sum\limits_{i=1}^j w_i) \chi - \sum\limits_{i =1}^{j-1} \chi_i - \sum\limits_{i = j+1}^{n-1} k_i + (j-1)r. 
 \end{split} \end{equation*}
This establishes the left hand side inequality of \eqref{eqn:1} for $j = 1, \dots, n-1$. Now, we prove the right hand side inequality of \eqref{eqn:1} i.e., 
\[ 
\chi_j \leq (\sum\limits_{i =1}^{j} w_i)  \chi - \sum\limits_{i=1}^{j-1} \chi_i + \sum\limits_{i=1}^{j}k_i + (j-1)r.
\]
From \eqref{inquality  the first when w stable vector bundle} and \eqref{inquality  the second when w stable vector bundle}, we get 
\begin{eqnarray*}
\sum\limits_{i=1}^{j}{\chi_i} \leq \chi(\sum\limits_{i=1}^{j}w_i) + \sum\limits_{i=1}^{j}k_i + (j-1)r.
\end{eqnarray*}
This implies 
\[
\chi_j \leq \chi (\sum\limits_{i=1}^{j}w_i) -\sum\limits_{i=1}^{j-1}\chi_i+ \sum\limits_{i=1}^{j}k_i + (j-1)r.
\]
The same proof with strict inequalities holds when $E$ is $w$-stable.
\end{proof}
\end{theorem} 

\begin{remark}
For $n=2$, the inequalities \eqref{eqn:1} coincide with those of \cite[Lemma 3.3]{Brivio-Favale-VB-2020}.
\end{remark}
 
At this stage it is natural to ask if the converse of the previous theorem holds, at least under some assumptions 
on $E_i$ (as it is done in \cite[Proposition 3.6]{Brivio-Favale-VB-2020} for the case $n=2$).   
We answer this in Proposition \ref{converse}. We first recall the following definition from  \cite{Narasimhan-Ramanan-1978}.

\begin{definition}
Let $G$ be a vector bundle on a smooth curve. For any integer $k$ we set 
\[
\mu_k(G) := \frac{\deg(G)+k}{\text{rk}(G)}.
\]
We say that $G$ is $(m,k)$-semistable (resp. stable) if for every subsheaf $F$ of $G$, 
\[
\mu_k(F) \leq \mu_{m-k}(G) \quad (\text{resp.} \; \mu_k(F)<\mu_{m-k}(G)).
\] 
\end{definition}
\noindent Note that if $G$ is $(m,k)$-semistable, then so is $G \otimes L$, for any line bundle $L$. 
\begin{proposition}\label{converse}
Let $E = E_u$ be the sheaf defined by 
$u = ([E_1],\dots,[E_n]; [\sigma_1],\dots, [\sigma_{n-1}]) \in  \mathbb{P}$ with $\text{rank}(\sigma_i) = k_i$. Let $\chi_i$ denote 
the Euler characteristic of $E_i$ and $w = (w_1,...,w_n)$ be a polarisation such that 
\begin{eqnarray}\label{hypothesis}
\chi_1 \leq w_1\chi + k_1, \;  \; \chi_n \leq w_n\chi + r \cr
\chi_i \leq w_i\chi + k_i + r \; (2 \leq i \leq n-1).
\end{eqnarray}
Assume that $E_1$ is $(0,k_1)$-semistable, $E_n$ is $(0,r)$-semistable and $E_i$ is $(0,k_i+r)$-semistable for $2\leq i \leq n-1$. Then $E$ is $w$-semistable.
\end{proposition}

\begin{proof}
Let $F$ be a subsheaf of $E$ such that the stalks of $F$ at the nodes $p_i$ are 
$\mathcal{O}_{p_i}^{s_i'} \oplus \mathcal{O}_{q_i}^{a_i} \oplus \mathcal{O}_{q_i'}^{b_i}$. To prove the theorem 
we need to show $\mu_w(F) \leq \mu_w(E)$. Let $s_i = s_i' + a_i$ for 
$1\leq i \leq n-1$ and $s_n = s_{n-1}'+b_{n-1}$. Note that 
$s_i' + b_i = s_{i+1}$ for $1\leq i \leq n-1$. We then see that multirank of $F$ is $(s_1,...,s_n)$. Now by 
construction of $E$, for each $1\leq i\leq n$, there exists a vector bundle $F_i$ such that $F_i \subset E_i$ and $F$ 
is the kernel of 
\[
\tilde{\sigma}|_{\oplus_{i=1}^n j_{i_*}(F_i)} : \oplus_{i=1}^n j_{i_*}(F_i) \rightarrow \oplus_{i=1}^{n-1} j_{p_{i_*}}(j_{p_i}^*(j_{p_{{i+1}_*}}E_{i+1})).
\]
From the commutative diagram (just before Definition \ref{def of E_u}), we get the following short exact sequence:
\begin{eqnarray}\label{short exact seq for converse}
0\rightarrow \oplus_{i=1}^n G_i \rightarrow F \rightarrow \oplus_{i=1}^{n-1} \mathbb{C}_{p_i}^{s_i'} \rightarrow 0,
\end{eqnarray}
where $G_i = \ker(\rho_i'|_{j_{i_*}(F_i)}, \sigma_i\circ \rho_i|_{j_{i_*}(F_i)}) \subset \tilde{K_i}$. Now since $E_1$ is $(0,k_1)$-semistable and $G_1$ 
is a subsheaf of $E_1$, we have
\[
\frac{\deg(G_1)}{s_1} \leq \frac{d_1-k_1}{r}.
\]
For $2\leq i \leq n-1$, $E_i$ is $(0,k_i+r)$-semistable implies that $E_i(-q_{i-1}')$ is also $(0,k_i+r)$-semistable. As $G_i$ is a subsheaf of $E_i(-q_{i-1}')$, for $2 \leq i \leq n-1$ we have
\[
\frac{\deg(G_i)}{s_i} \leq \frac{d_i-k_i-2r}{r}.
\]
Also, $E_n$ is $(0,r)$-semistable implies $E_n(-q_{n-1}')$ is $(0,r)$-semistable. Now, since $G_n$ is a subsheaf of $E_n(-q_{n-1}')$, we have
\[
\frac{\deg(G_n)}{s_n} \leq \frac{d_n-2r}{r}.
\]
Now from \eqref{short exact seq for converse}, we have 
\begin{eqnarray*}
\chi(F) &=& \chi(\oplus_{i=1}^n G_i) + \sum_{i=1}^{n-1}s_i' \\
&=& \sum_{i=1}^{n}(\deg(G_i) + s_i(1-g_i)) + \sum_{i=1}^{n-1}s_i' \\
&\leq& s_1w_1\frac{(d_1-k_1)+r(1-g_1)}{w_1r} + s_2w_2\frac{(d_2-k_2-2r)+r(1-g_2)}{w_2r} + \cdots  \\
&& +\; s_{n-1}w_{n-1}\frac{(d_{n-1}-k_{n-1}-2r)+r(1-g_{n-1})}{w_{n-1}r} 
+ s_nw_n\frac{(d_n-2r)+r(1-g_n)}{w_nr} \\
&&  + \sum_{i=1}^{n-1}s_i' \\
&=& s_1w_1\mu_w(\tilde{K_1}) + s_2w_2\mu_w(\tilde{K_2}) - s_2 +...+ s_nw_n\mu_w(\tilde{K_n}) - s_n + \sum_{i=1}^{n-1}s_i' \cr
&=& \sum_{i=1}^n s_iw_i\mu_w(\tilde{K_i}) + \sum_{i=1}^{n-1}(s_i'-s_{i+1}) \cr
&\leq& \sum_{i=1}^n s_iw_i\mu_w(E) + \sum_{i=1}^{n-1}(s_i'-s_{i+1}), 
\end{eqnarray*}
where the last inequality follows from the inequalities \eqref{hypothesis}.
Using this and the fact that $\sum_{i=1}^{n-1}(s_i'-s_{i+1}) < 0$, we have 
\begin{eqnarray*}
\mu_w(F) &\leq& \mu_w(E) + \frac{\sum_{i=1}^{n-1}(s_i'-s_{i+1})}{\sum_{i=1}^ns_iw_i} \cr
&\leq&  \mu_w(E).
\end{eqnarray*}
Note that if any of $E_i$ is $(0,k_i)$-stable or $E_n$ is $(0,r)$-stable, the above inequality is strict and hence $E$ is $w$-stable.
\end{proof}

\begin{remark}
Note that the inequalities \eqref{hypothesis} imply the inequalities \eqref{eqn:1}, but they are not equivalent 
to each other. The above Proposition is only a partial converse to Theorem 
\ref{lemma regarding the existance of the bigas inequality}.
\end{remark}

\section{Rationality of Moduli space of vector bundles}
In this section, we prove the main results of this paper. Let $ \mathbb{P}$ be the fibre product on $\mathcal{U}_{C_1}(r, d_1) \times \cdots \times \mathcal{U}_{C_n}(r, d_n)$  (defined in Theorem \ref{theorem on exists on poincare bundle on chain-like curve}). For $1\leq k \leq r-1$, let $\mathscr{B}_{i,k} \subset \mathbb{P}$ be such that
\begin{align*}
\mathscr{B}_{i,k} \cap \pi^{-1}([E_1],\dots, [E_n]) = \lbrace ([\sigma_1], \dots , [\sigma_{n-1}]) \in  \mathbb{P} \vert rk(\sigma_i) \leq k \rbrace,
\end{align*} 
is a proper closed subvariety of $ \mathbb{P}$. We then define $\mathscr{B}_{k} := \cup_{i=1}^{n-1}\mathscr{B}_{i,k}$.
\begin{definition}
Let  $\mathscr{U}$ be the set defined by the complement of $\mathscr{B}_{r-1}$ in $\mathbb{P}$. 
\end{definition}
Note that $\mathscr{U}$ 
is the open set for which $E_u$ is a vector bundle.

\begin{remark}\label{remark to compute the dim}

 Let $\pi_{\vert_\mathscr{U}}$ be the restriction of $\pi$ to $\mathscr{U}$. Then 
$\pi_{\vert_\mathscr{U}} : \mathscr{U} \rightarrow \mathcal{U}_{C_1}(r, d_1) \times \cdots \times \mathcal{U}_{C_n}(r, d_n)$
is a bundle, and by construction the fibre of $\pi_{\vert_\mathscr{U}}$ at $([E_1], \dots , [E_n])$  is 
\[
 \text{PGL}(E_{1,q_1}, E_{2, q_1'}) \times \cdots  \times \text{PGL}(E_{n-1,q_{n-1}}, E_{n, q_{n-1}'}).
 \]
Since the rank of $ \text{PGL}(E_{i,q_i}, E_{i+1, q_i'})$ is $r^2-1$, we have the following:
\[ 
\begin{split}
      \dim (\mathscr{U}) 
      & = r^2\left(\sum\limits_{i=1}^n g_i -n\right) + n + \text{dimension of fibre} \\
      & = r^2\left(\sum\limits_{i=1}^n g_i -n\right) + n + (r^2-1)(n-1) \\
      & =   r^2\left(\sum\limits_{i=1}^n g_i -1\right) + 1. 
\end{split}
\]
\end{remark} 

For $\chi  = \sum\limits_{i=1}^n d_i + r (1 - \sum\limits_{i =1}^n g_i)$, let $\mathcal{U}_C(w,r,\chi)_{d_1, \dots, d_n}$ be the irreducible component of the moduli space of $w$-stable pure sheaves of dimension one on a chain-like curve $C$ with rank $r$ and Euler characteristic $\chi$ corresponding to multidegree $(d_1, \dots,  d_n)$. Let $\mathcal{V}_C(w,r, \chi)_{d_1 ,\dots , d_n} \subset \mathcal{U}_C(w,r, \chi)_{d_1 ,\dots , d_n}$ be the subset parametrizing classes of vector bundles. This is an open subset 
and is therefore irreducible.

\begin{theorem}\label{theorem for the first birational map}
Let $C$ be a chain-like curve having $n$ smooth components $C_i$ with $(n-1)$ nodes $p_i$ and genus $g_i \geq 2$. Fix $r \geq 2$, and $d_i \in \mathbb{Z}$, for $i = 1,\dots, n$. Let $\chi_i = d_i + r (1 -g_i)$,  $\chi = \sum\limits_{i=1}^{n} d_i + r(1 - \sum\limits_{i=1}^{n} g_i)$ and $w = (w_1, \dots, w_n)$ be a polarization. Assume that $(\chi_1, \dots , \chi_n)$ and $w= (w_1 , \dots, w_n)$ satisfy the inequalities

\[
(\sum\limits_{i=1}^j w_i)\chi - \sum\limits_{i=1}^{j-1}\chi_i + r(j-1) < \chi_j < (\sum\limits_{i=1}^j w_i)\chi - \sum\limits_{i=1}^{j-1}\chi_i + rj,
\]

for $1\leq j \leq n-1$. Then the map
 
\[
\varphi :  \mathbb{P}\dashrightarrow \mathcal{U}_{C}(w,r, \chi)_{d_1, \dots , d_n}
\]
defined by
\[
u \mapsto [E_u]
\]
is birational. Moreover, $\varphi_{\vert_\mathscr{U}}$ is an injective morphism on $\mathscr{U}$.
\begin{proof}
Let $u = ([E_1],\dots, [E_n]; [\sigma_1], \dots , [\sigma_{n-1}]) \in \mathbb{P}$, and let 
$E_u$ be the sheaf defined by $u$ (see Section \ref{section on Chain-like curves}). By definition of $\mathscr{U}$, $\varphi_{\vert_\mathscr{U}}$ is well defined on $\mathscr{U}$.
\\ \\
\underline{(a) To prove $\varphi$ is injective on $\mathscr{U}$:} \\
Let $u  = ([E_1],\dots, [E_n]; [\sigma_1], \dots , [\sigma_{n-1}])$ and $u'  = ([E_1'],\dots, [E_n']; [\sigma_1'], \dots , [\sigma_{n-1}'])$ $\in \mathscr{U}$ such that  $[E_u] = [E_{u'}] \in \mathcal{U}_{C}(w,r, \chi)_{d_1, \dots , d_n}$. Therefore, $E_u$ and $E_{u'}$ are $S_w$-equivalent.    
Since  the $S_w$-equivalent bundles $E_u$ and $E_{u'}$ are $w$-stable, they are also isomorphic. Let $\alpha : E_u \rightarrow E_{u'}$ be an isomorphism. Therefore, we have an induced isomorphism $\alpha_i : E_i \rightarrow E_i'$, for each $i = 1, \dots , n$. Now we have the following diagram,
\[
\xymatrix{
E_{i, q_i} \ar[rr]^{\sigma_i} \ar[dd]^{(\alpha_i)_{q_i}} && E_{i+1, q_i'} \ar[dd]^{(\alpha_{i+1})_{q_i'}}
\\ \\
E_{i. q_i}' \ar[rr]_{\sigma_i'} && E_{i+1, q_i'}
}
\]
Since $E_i$ and $E_i'$ are stable bundles, $Hom (E_i, E_i') \cong \mathbb{C}$. Note that $(\alpha_i)_{q_i}$ and $(\alpha_{i+1})_{q_i'}$ are scalar linear maps. By using the commutative diagram above, $\sigma_i'$ is a scalar multiple of $\sigma_i$ for each $i = 1, \dots, n-1$. Therefore, $[u]$ = $[u']$ and hence $\varphi$ is injective on $\mathscr{U}$.
\\
\\
\underline{(b) To prove $\varphi_{|_\mathscr{U}}$ is a morphism}:
Assume first that $(r,d_i)=1$ for all $i$. To show this is a morphism, it  suffices to show that $\varphi$ is regular at every point of  
$\mathscr{U}$. Let $u_0 \in \mathscr{U}$. Indeed, we find  an open set $W \subset \mathscr{U}$ 
containing $u_0$ and a vector bundle $\mathcal{E}$ on $\mathscr{U} \times C$ such that 
\[
\mathcal{E}_{|_{x \times C}} = \varphi(x),
\]
for all $x \in W$. \\
\underline{Step-1}: There are two sheaves $\mathcal{Q}$ and $\mathcal{R}$ on 
$\mathscr{U} \times C$ such that 
\[
\mathcal{Q}_{|_{x \times C}} \cong \bigoplus_{i=1}^{n}j_{i*}(E_i) \quad \text{and}\quad \mathcal{R}_{|_{x \times C}} \cong \bigoplus_{i=1}^{n-1} j_{p_i *}(j_{p_i}^*(j_{i+1 *}(E_{i+1}))).
\]
Consider the following diagram for $q_i$, $\iota_i$ ($1\leq i \leq n-1$) and also for 
${q}_{i-1}'$, ${\iota}_{i-1}'$ ($2\leq i \leq n$), which we omit. 
For simplicity, we denote $\mathcal{U}_{C_i}(r, d_i)$ by $\mathcal{U}_{C_i}$ in the figure below: 
\[
\begin{tikzcd}[column sep={4.5em,between origins},row sep=2em]
&
\mathscr{U}
  \arrow[rr, leftarrow, "\cong"]
  \arrow[dl,swap, "J_{p_i}"]
  \arrow[dd,swap,]
&&
\mathscr{U} \times p_i
\\
\mathscr{U} \times C  \arrow[dd,swap,"\Pi_\mathscr{U}"]
\\
&
\mathcal{U}_{C_1} \times \cdots \times \mathcal{U}_{C_n}
  \arrow[rr,rightarrow, "\pi_i"]
  \arrow[dl,swap,]
&&
\mathcal{U}_{C_i}
  \arrow[dl,rightarrow,] \arrow[rr, leftarrow, "\cong"] 
&& 
\mathcal{U}_{C_i} \times q_i
 \arrow[dl,rightarrow, "id \times \iota_i"]
\\
\mathcal{U}_{C_1} \times \cdots \times \mathcal{U}_{C_n}  \times C
  \arrow[rr,rightarrow, "\Pi_i"]
&&
\mathcal{U}_{C_i} \times C
 \arrow[rr, hookleftarrow, "J_i"]
  && 
  \mathcal{U}_{C_i} \times C_i
\end{tikzcd}
\]
Notational convenience: $\Pi_{i} = \pi_i \times \text{id}_C$, $\Pi_\mathscr{U} = \pi_{|_\mathscr{U}} \times \text{id}_C$, $J_{i} = \text{id}_{\mathscr{U}_{C_i}} \times j_i$.\\
Consider 
\begin{align}
\mathcal{Q}_i := \Pi_\mathscr{U}^*(\pi_i^* (J_{i*}\mathcal{P}_i)))\quad \text{and} \quad \mathcal{R}_i = J_{p_i *} (J_{p_i}^* (\mathcal{Q}_{i +1})).
\end{align}
Then set
\begin{align*}
\mathcal{Q} = \bigoplus_{i = 1}^n \mathcal{Q}_i \quad \text{and} \quad \mathcal{R} = \bigoplus_{i = 1}^{n-1} J_{p_i *} (J_{p_i}^* (\mathcal{Q}_{i +1})).
\end{align*}
Note that, $\text{Supp}(\mathcal{R})= \bigcup\limits_{i= 1}^{n-1} (\mathscr{U} \times p_i)$.
\\
\underline{Step-2}: There exists an open set $W$ of $\mathscr{U}$ containing $u_0$ and an onto map of sheaves
\begin{align}
\Omega_W : \mathcal{Q}_{|_{W \times C}} \rightarrow \mathcal{R}_{|_{W \times C}} 
\end{align}
such that $\text{Ker}(\Omega_W)= \mathcal{E}$.
Let $\pi : \mathbb{P} \rightarrow \mathcal{U}_{C_1}(r, d_1) \times \cdots \times \mathcal{U}_{C_n}(r, d_n)$ be the product of projective bundles (see Theorem \ref{theorem on exists on poincare bundle on chain-like curve}) and let 
$P_i$ denote the projection onto its $i$th factor $\mathbb{P}(\mathcal{F}_i)$. Let $\bigoplus\limits_{i=1}^{n-1}P_i^*(\mathcal{O}_{\mathbb{P}(\mathcal{F}_i)}(-1))$ be the direct sum of pullback of tautological line bundles $\mathcal{O}_{\mathbb{P}(\mathcal{F}_i)}(-1)$. Let $u = ([E_1],\dots, [E_n]; [\sigma_1], \dots , [\sigma_{n-1}])$, then the fibre of $\pi$ at $u$ is
\[
\text{Span}(\sigma_1, \dots , \sigma_{n-1}) \subset \bigoplus_{i =1}^{n-1} \text{Hom}(E_{i, q_i}, E_{i+1, q_i'}).
\]
Thus, given $u_0 \in \mathbb{P}$, there is an open set $W$ containing $u_0$ and set of sections $s_i \in \bigoplus\limits_{i=1}^{n-1}P_i^*(\mathcal{O}_{\mathbb{P}(\mathcal{F}_i)}(-1)(W))$ such that  $s_i(u) \neq 0$ for all $u \in W$.
This implies that the stalk of the map
\begin{align}\label{stalk map}
s_i : \Pi_\mathscr{U}^*(\pi_i^*(\iota_i^*(\mathcal{P}_i)))\vert_{W} \rightarrow \Pi_\mathscr{U}^*(\pi_i^*({\iota'}_{i+1}^*(\mathcal{P}_{i +1})))\vert_{W}
\end{align}
i.e., 
\[
(s_i)_u : E_{i, q_i} \rightarrow E_{i+1, q_i'}
\]
is an isomorphism, and $[(s_1)_u , \dots , (s_{n-1})_u] = ([\sigma_1], \dots , [\sigma_{n-1}])$ in $\mathbb{P}(\oplus_{i =1}^{n-1} \text{Hom}(E_{i, q_i}, E_{i+1, q_i'}))$. For convenience, we use the notation $\mathcal{T}_{i}$ (respectively $\mathcal{T}_{i+1}'$) for $\Pi_\mathscr{U}^*(\pi_i^*(\iota_i^*(\mathcal{P}_i)))$ (respectively $\Pi_\mathscr{U}^*(\pi_{i+1}^*({\iota'}_{i+1}^*(\mathcal{P}_{i+1})))$) for $i = 1, \dots , n-1$. Then $J_{p_i}^*(\mathcal{Q}_i) = \mathcal{T}_i$ for $i = 1, \dots , n-1$. Also, \eqref{stalk map} can be rewritten as 
\[
s_i : {\mathcal{T}_i}\vert_W \rightarrow {\mathcal{T}_{i+1}'}\vert_W.
\]
Define
\[
s_i - \text{id} : {\mathcal{T}_i}\vert_W \oplus {\mathcal{T}_{i+1}'}\vert_W \rightarrow {\mathcal{T}_{i+1}'}\vert_W,
\]
for $i = 1, \dots n-1$. Consider the following commutative diagram 
\[
\begin{tikzcd}[column sep={15.5em,between origins},row sep=3em]
{\mathcal{Q}_i \oplus \mathcal{Q}_{i+1}}\vert_{W \times C} \arrow{r}{\Omega_W^i} \arrow[swap]{d} & {\mathcal{R}_i}\vert_{W \times C} \arrow{d}{\varrho_x^g} \\
J_{p_i}^*(({\mathcal{Q}_i \oplus \mathcal{Q}_{i+1}})\vert_{W \times C}) \arrow{r} \arrow{d} & J_{p_i}^*({\mathcal{R}_i}\vert_{W \times C}) \arrow{d} \\
{\mathcal{T}_i}\vert_W \oplus {\mathcal{T}_{i+1}'}\vert_W \arrow{r}  & {\mathcal{T}_{i+1}'}\vert_W.
\end{tikzcd}
\]
We have the composite map 
\[
{\mathcal{Q}_i \oplus \mathcal{Q}_{i+1}}\vert_{W \times C} \rightarrow J^*_{p_i}({\mathcal{R}_i}\vert_{W \times C}),
\]
and since $\text{supp}({\mathcal{R}_i}_{\vert_{W \times C}})$ is $W \times p_i$, this is actually defined onto ${\mathcal{R}_i}_{\vert_{W \times C}}$ to give a map $\Omega_W^i$ for $i = 1, \dots, n-1$. Set $\Omega_W := \bigoplus_{i=1}^{n-1} \Omega_W^i$, then the required vector bundle is $\mathcal{E} = \text{Ker}(\Omega_W)$. By construction $\varphi(\mathscr{U}) \subset \mathcal{V}_C(w,r, \chi)_{d_1 , \dots , d_n}$, and it is equal to the open subset of $w$-semistable vector bundles whose restrictions are stable.
By  Remark \ref{remark to compute the dim},
\[
\dim(\varphi(\mathscr{U})) = \dim(\mathscr{U}) = r^2(\sum\limits_{i=1}^n g_i -1) +1,
\]
which  is equal to $\dim(\mathscr{U}(r, w, \chi)_{d_1, \dots , d_n})$. Therefore, $\varphi$ is a dominant map and hence $\varphi_{|_\mathscr{U}}$ is a birational morphism.

Without the coprimality assumption, it is sufficient to prove that $\varphi|_{\mathscr{U}}$ is a morphism at the Quot 
scheme level. The same proof now works.
\end{proof}

\end{theorem}

Let $C$ be a smooth complex projective curve of genus $g$. For a line bundle $L$ of degree $d$ on $C$, let $\mathcal{SU}_C(r, L)$ be the moduli space of semistable vector bundles on $C$ of rank $r$ and fixed determinant $L$. If $(r, d)=1$, $\mathcal{SU}_C(r, L)$ is known to be an irreducible projective subvariety of $\mathcal{U}_C(r, d)$. \\ \\
Let $L = (L_1, \dots , L_n)$ be a line bundle on a chain-like curve $C$, where $L_{\vert_{C_i}} = L_i$ be a line bundle on $C_i$ with $\deg{L_i}=d_i$. Let $\mathcal{SU}_{C_i}(r, L_i)$ be the fixed determinant moduli space of rank $r$ vector bundles on $C_i$ with determinant $L_i$. Consider the inclusion 
\begin{eqnarray}
\mathcal{SU}_{C_1}(r, L_1) \times \cdots \times \mathcal{SU}_{C_n}(r, L_n) \hookrightarrow  \mathcal{U}_{C_1}(r, d_1) \times \cdots \times \mathcal{U}_{C_n}(r, d_n).
\end{eqnarray} 
We have sheaves $\mathcal{F}_1, \dots , \mathcal{F}_{n-1}$ on the product $\mathcal{U}_{C_1}(r, d_1) \times \cdots \times \mathcal{U}_{C_n}(r, d_n)$ (see Theorem \ref{theorem on exists on poincare bundle on chain-like curve}). We denote by 
$\mathcal{F}_{i, L}$, the sheaves obtained by restricting the sheaves $\mathcal{F}_i$ on the product $\mathcal{SU}_{C_1}(r, L_1) \times \cdots \times \mathcal{SU}_{C_n}(r, L_n)$. Let
\begin{align*}
\mathscr{U}_L  := \mathscr{U} \cap \mathbb{P}_L,
\end{align*}
where $\mathbb{P}_L$ denotes the fibre product of the $\mathbb{P}(\mathcal{F}_{i,L})$. In the non-coprime case, we need to use Remark \ref{birational model without coprimality conditions}.
Let $\overline{\mathcal{SU}_C(w,r,\chi,L)}$ denote the closure of the collection of $w$-semistable vector bundles of rank $r$ in $\mathcal{U}_C(w,r,\chi)_{d_1,\dots,d_n}$ with determinant $L$, where $d_i = \deg(L_{\vert_{C_i}})$. \\
Let $E$ be a pure sheaf of dimension one on $C$ with multirank $(r_1,\dots,r_n)$. Tensoring the exact sequence \eqref{canonical exact sequence} by $E$, we get
\begin{equation}
0 \rightarrow E \rightarrow \bigoplus_{j=1}^n E_j \rightarrow \mathcal{T}_E \rightarrow 0,
\end{equation}
where $\mathcal{T}_E$ is supported only at nodes $p_i$. Moreover $h^0(\mathcal{T}_E)=  \sum\limits_{i=1}^{n-1} t_{p_i}$, where $t_{p_i}$ is the rank of the free part of the stalk $E_{p_i}$. Let $L= (L_1, \dots , L_n)$ be a line bundle on $C$. Then we have
\begin{eqnarray*}
\chi(E \otimes L) = \sum\limits_{i =1}^n \chi(E_i \otimes L)- \sum\limits_{i =1}^{n-1} t_{p_i}.
\end{eqnarray*}
The projection formula implies that $\chi(E_i \otimes L) = \chi(E_i \otimes L_i) = \chi(E_i) + r_i \deg(L_i)$. Therefore, 
\begin{eqnarray*}
\chi(E \otimes L) = \chi(E) + \sum\limits_{i=1}^{n} r_i \deg(L_i).
\end{eqnarray*}
In particular, if for each $i$, $\deg(L_i) = 0$, then $\chi(E \otimes L) = \chi(E)$. Therefore, in this case,
\begin{eqnarray}\label{equation when slop is same tensor with line bundle of deg 0}
\mu_w(E \otimes L)= \mu_w(E)
\end{eqnarray}
with respect to any polarization $w$.

\begin{lemma}\label{lemma on w-stability of vector bundle tensoring with line bundle}
Let $E$ be a pure sheaf of dimension one on a chain-like curve $C$, let $L$ be line bundle on $C$ with $\deg(L_{|_{C_i}}) = 0$. Suppose that $w$ is a polarization. Then $E$ is $w$-semistable (resp. $w$-stable) if and only if $E \otimes L$ is $w$-semistable (resp. $w$-stable).
\begin{proof}
Let $F$ be a proper subsheaf of $E \otimes L$. Then $F \otimes L^{-1}$ is a proper subsheaf of $E$. Since $E$ is $w$-semistable (resp. $w$-stable), $\mu_w(F \otimes L^{-1}) \leq \mu_w(E)$ (resp. $\mu_w(F \otimes L^{-1}) < \mu_w(E)$). By equation \eqref{equation when slop is same tensor with line bundle of deg 0}, $\mu_w(F) \leq \mu_w(E \otimes L)$ (resp. $\mu_w(F ) < \mu_w(E\otimes L)$). And vice versa.
\end{proof}
\end{lemma}

\begin{remark}
Lemma \ref{lemma on w-stability of vector bundle tensoring with line bundle} may not be true if the line bundle $L$ on $C$ satisfies $\deg ( L_{|_{C_i}}) \neq 0$, for some $i$ (See \cite[Remark 4.7]{Suhas-susobhan-amit-2021}).  
\end{remark}

\begin{corollary}\label{main corollory}
Under the same hypothesis of Theorem \ref{theorem for the first birational map}, the map
\begin{eqnarray*}
\varphi_L : \mathbb{P}_L \dashrightarrow \overline{\mathcal{SU}_C(w,r,\chi,L)}
\end{eqnarray*}
is birational, where the restriction ${\varphi_L}_{\vert_{\mathscr{U}_L}}$ of $\varphi_L$ to $\mathscr{U}_L$ is injective.
\begin{proof}
Note that ${\varphi_L}_{\vert_{\mathscr{U}_L}}$ is a morphism and its image is the set $\text{Im}({\varphi_L}_{\vert_{\mathscr{U}_L}}) = \lbrace E \in \mathcal{V}_C(w,r, \chi)_{d_1, \dots , d_n} | [E_{C_i}] \in \mathcal{SU}_{C_i}(r, L_i) \rbrace $. In particular, $\text{Im}({\varphi_L}_{\vert_{\mathscr{U}_L}}) \subset \mathcal{SU}_{C}(w,r, \chi)$. 
Consider the map
\begin{eqnarray}
\psi : \mathcal{V}_C (w,r, \chi)_{d_1, \dots , d_n} \rightarrow \text{Pic}^{d_1}(C_1) \times \cdots \times \text{Pic}^{d_n}(C_n) 
\end{eqnarray} 
defined by 
\begin{eqnarray*}
E \mapsto (\det(E_{|_{C_1}}), \dots , \det(E_{|_{C_n}})).
\end{eqnarray*}
We then have the following commutative diagram:
\begin{eqnarray*}
\begin{tikzcd}[column sep={17em,between origins},row sep=4em]
\mathscr{U} \arrow{r}{\varphi} \arrow[swap]{d}{\pi_{\mathscr{U}}} & \mathcal{V}_C (w,r, \chi)_{d_1, \dots , d_n} \arrow{d}{\psi} \\
\mathcal{U}_{C_1}(r, d_1) \times \cdots \times \mathcal{U}_{C_n}(r, d_n) \arrow{r}{\det \times \cdots \times \det} & \text{Pic}^{d_1}(C_1) \times \cdots \times \text{Pic}^{d_n}(C_n). 
\end{tikzcd}
\end{eqnarray*}
Notice that $\text{Im}(\varphi_L) \subset \psi^{-1}(L_1, \dots , L_n)$, where the map $\det \times \cdots \times \det : \mathcal{U}_{C_1}(r, d_1) \times \cdots \times \mathcal{U}_{C_n}(r, d_n) \rightarrow \text{Pic}^{d_1}(C_1) \times \cdots \times \text{Pic}^{d_n}(C_n)$ is defined by $(E_1, \dots , E_n) \mapsto (\det(E_1), \dots, \det(E_n))$. Also, note that the map $\psi$ is onto, since $(\det \times \cdots \times \det) \circ \pi_{\mathscr{U}}$ is onto.

 We now claim that any two fibres of $\psi$ are isomorphic. If $(L_1, \dots , L_n), (L'_1, \dots , L'_n) \in \text{Pic}^{d_1}(C_1) \times \cdots \times \text{Pic}^{d_n}(C_n)$, then there exist line bundles $\zeta_i \in \text{Pic}^{0}(C_i)$  such that $L_i \otimes \zeta_i^r \cong {L'_i}$. Let $\zeta := (\zeta_1, \dots , \zeta_n)$ be the unique line bundle on $C$ obtained by gluing the line bundles $\zeta_1, \dots, \zeta_n$. If a rank $r$ vector bundle $E$ is $w$-stable, Lemma \ref{lemma on w-stability of vector bundle tensoring with line bundle} implies that $E \otimes \zeta$ is $w$-stable, and vice versa. Therefore, $E \mapsto E \otimes \zeta$ gives an isomorphism from $\psi^{-1}(L_1, \dots , L_n)$ to $\psi^{-1}(L'_1, \dots , L'_n)$. 
 
 Now to see that the fibres are irreducible, note that, since $\varphi (\mathscr{U})$ is dense in the irreducible variety 
 $\mathcal{V}_C(w,r, \chi)$, the same holds for $\varphi_L({\mathscr{U}_L}) \subset \psi^{-1}(L_1,\dots,L_n)$ for general 
 $(L_1,\dots,L_n)$. Therefore, the general fibre  $\psi^{-1}(L_1,\dots,L_n)$ is irreducible. Since all fibres are isomorphic, 
 they are irreducible.
\end{proof}

\end{corollary}
 \begin{theorem}\label{theorem on rationality}
  Let $C$ be a chain-like curve as mentioned above. Fix $r \geq 2$, and $d_i \in \mathbb{Z}$ with $(r, d_i) = 1$ for $i = 1,\dots, n$. Let $\chi_i = d_i + r (1 -g_i)$ and $\chi = \sum\limits_{i=1}^{n} \chi_i - (n-1)r$. For any line bundle $L$ of multidegree $(d_1,\dots,d_n)$ and any polarization $w = (w_1, \dots, w_n)$ satisfying the conditions

\[
(\sum\limits_{i=1}^j w_i)\chi - \sum\limits_{i=1}^{j-1}\chi_i + r(j-1) < \chi_j < (\sum\limits_{i=1}^j w_i)\chi - \sum\limits_{i=1}^{j-1}\chi_i + rj,
\]

for $1\leq j \leq n-1$, $\overline{\mathcal{SU}_C(w,r,\chi,L)}$ is a rational variety.

\begin{proof}
Since $d_i$ and $r$ are coprime for every $i$, by \cite{King-Schofield-1999}, the moduli space $\mathcal{SU}_{C_i}(r,L_i)$ is a rational variety. So  the product $\mathcal{SU}_{C_1}(r,L_1)\times \cdots \times \mathcal{SU}_{C_n}(r,L_n)$ is rational. Also, we know that $\mathscr{U}_L$ is $\mathbb{P}^{r^2-1}\times \cdots \times \mathbb{P}^{r^2-1}$-bundle over  $\mathcal{SU}_{C_1}(r,L_1)\times \cdots \times \mathcal{SU}_{C_n}(r,L_n)$. Therefore, it is rational too. So, in view of Corollary \ref{main corollory}, $\overline{\mathcal{SU}_C(w,r,\chi,L)}$ is a rational variety.
\end{proof}
\end{theorem}

\section{Existence of polarization}\label{Existence of polarization}
Let $r\geq 2$ be an integer and $(\chi_1,\dots,\chi_n)\in \mathbb{Z}^n$ with $\chi = \sum\limits_{i=1}^n \chi_i - (n-1)r$. In this section,  we show the existence of a polarization $w = (w_1,\dots, w_n)$ satisfying the inequalities in Theorem \ref{Section-5}, under some conditions on $\chi_i$.

\begin{theorem}
Let $r \geq 2$ be an integer. Then, there exists a non-empty subset 
$\mathcal{W}_{r, n} \subset \mathbb{Z}^n$ such that for any n-tuple 
$(\chi_1,...,\chi_n) \in \mathcal{W}_{r, n}$, we can find a polarization $w = (w_1,w_2,...,w_n)$ satisfying the conditions 
\begin{align}\label{eqn:3}
(\sum\limits_{i =1}^{j} w_i)  \chi - \sum\limits_{i=1}^{j-1} \chi_i +(j-1)r < \chi_j  < (\sum\limits_{i =1}^{j} w_i)  \chi - \sum\limits_{i=1}^{j-1} \chi_i + jr,
\end{align}
for every $1 \leq j \leq n$, where $\chi=\sum\limits_{i=1}^{n}\chi_{i}-(n-1)r$.
\end{theorem}

\begin{proof}
$\underline{Case\; 1: \; \chi = 0}.$ 
In this case, \eqref{eqn:3} becomes
\[
- \sum\limits_{i=1}^{j-1} \chi_i +(j-1)r < \chi_j < - \sum\limits_{i=1}^{j-1} \chi_i + jr,
\]
which can be rewritten for each $1 \leq j \leq n$ as
\begin{eqnarray}\label{eqn:3.1}
(j-1)r < \sum\limits_{i=1}^{j} \chi_i < jr.
\end{eqnarray}
Therefore, if we assume that $\chi_i$ ($1\leq i \leq n$) are such that the inequalities in \eqref{eqn:3.1} 
above with $\sum\limits_{i=1}^{n} \chi_i = (n-1)r$ are satisfied, then for any choice of $w_i \in (0,1) \cap \mathbb{Q}$ with $\sum\limits_{i=1}^{n}w_i = 1$, \eqref{eqn:3} holds.
\\
$\underline{Case\; 2 : \;\chi <0}.$ We deal with this case under the assumption that 
$\chi_i < 0$ for each $i$ and that $\chi = \sum\limits_{i=1}^n\chi_i - (n-1)r$. 
Then notice that for each $i$, $0< \frac{\chi_i}{\chi} < 1$, and for $j \in \{1,\dots,n-1\}$, 
\[
0 < \frac{ \sum\limits _{i=1}^j\chi_i - r(j-1)}{\chi} < \frac{\sum\limits_{i=1}^j\chi_j-rj}{\chi} < 1. 
\]
We prove this case by applying induction on $n$. To this end, we will first prove the existence of $w_1 \in (0,1) \cap \mathbb{Q}$ such that 
 \begin{equation} \label{first inequality for w1}
    w_1\chi < \chi_1 < w_1\chi + r,
 \end{equation}
 and then by induction, prove the existence of remaining $w_i$ satisfying \eqref{eqn:3}. 
Rewriting \eqref{first inequality for w1}, we get
\begin{equation}\label{another form of inequalities for w1}
    \chi_1 - r < w_1\chi < \chi_1.  
  \end{equation}
  Multiplying the inequalities (\ref{another form of inequalities for w1}) by ($\frac{1}{\chi}$) and using the hypothesis that $\chi < 0$, we get 
  \begin{equation}\label{third inequality for w1}
    \frac{\chi_1 - r}{\chi} > w_1 > \frac{\chi_1}{\chi}.
  \end{equation}
  We know that the numbers $\frac{\chi_1 - r}{\chi}$ and $\frac{\chi_1}{\chi}$ are between $0$ and $1$. Therefore, it is always possible to choose $w_1 \in (0,1) \cap \mathbb{Q}$ satisfying inequalities (\ref{first inequality for w1}).\\
  
  \noindent Now assume that the statement of the lemma hold for a positive integer $1 \leq j < n-1$. That is, we can choose
   $w_i$'s from $(0,1) \cap \mathbb{Q}$ such that $\sum\limits_{i=1}^jw_i \in (0,1) \cap \mathbb{Q}$ and
   \begin{equation*}
     \frac{\sum\limits _{i=1}^j\chi_i - rj}{\chi} > \sum\limits_{i=1}^{j}w_i > \frac{\sum\limits_{i=1}^j\chi_i - r(j-1)}{\chi}. 
  \end{equation*}
  Equivalently, \vspace{-0.03in}
  \begin{equation*}
  (\sum\limits_{i=1}^jw_i)\chi - \sum\limits_{i=1}^{j-1}\chi_i + r(j-1) <
  \chi_j < (\sum_{i=1}^jw_i)\chi - \sum\limits_{i=1}^{j-1}\chi_i + rj.
 \end{equation*}
  We now prove that, there exists $w_{j+1} \in (0,1) \cap \mathbb{Q}$, such that $\sum\limits_{i=1}^{j+1}w_i \in (0,1) \cap \mathbb{Q}$ and
  \begin{equation}\label{inequalities for w_(i+1)}
  (\sum\limits_{i=1}^{j+1}w_i)\chi - \sum\limits_{i=1}^{j}\chi_i + rj <
  \chi_{j+1} < (\sum\limits_{i=1}^{j+1}w_i)\chi - \sum\limits_{i=1}^{j}\chi_i + r(j+1).
 \end{equation}
Proceeding in the same way as we did for $w_1$, the inequality (\ref{inequalities for w_(i+1)}) is equivalent to 
 \begin{equation}\label{second inequality for w_(i+1)}
   \frac{\sum\limits_{i=1}^{j+1}\chi_i - r(j+1)}{\chi} > \sum\limits_{i=1}^{j+1}w_i >  \frac{\sum\limits_{i=1}^{j+1}\chi_i - rj}{\chi}.   
 \end{equation}
Now, $\chi_i < 0$ gives $\chi <  \sum\limits_{i=1}^{j+1}\chi_i - r(j+1)$, which in turn is strictly less than 
$\sum\limits_{i=1}^{j+1}\chi_i - rj$. Therefore, we see that $\frac{\sum\limits_{i=1}^{j+1}\chi_i - r(j+1)}{\chi}$ and 
$\frac{\sum\limits_{i=1}^{j+1}\chi_i - rj}{\chi}$ both lie in $(0,1) \cap \mathbb{Q}$. Therefore from 
\eqref{second inequality for w_(i+1)}, $\sum\limits_{i=1}^{j+1}w_i \in (0,1) \cap \mathbb{Q}$.
Choose any $w_{j+1}$ satisfying (\ref{second inequality for w_(i+1)}); it is enough to prove that $w_{j+1} > 0$. The inequalities in (\ref{second inequality for w_(i+1)}) are equivalent to
\begin{equation}\label{fourth inequlity for w_(i+1)}
     \frac{\sum\limits_{i=1}^{j+1}\chi_i - r(j+1)}{\chi} - \sum\limits_{i=1}^jw_i > w_{j+1} >  \frac{\sum\limits_{i=1}^{j+1}\chi_i - rj}{\chi}-\sum\limits_{i=1}^jw_i. 
     \end{equation}
Since $\sum\limits_{i=1}^jw_i < \frac{\sum\limits_{i=1}^j\chi_i - rj}{\chi}$, we have \vspace{-0.05in}
\begin{eqnarray*}
      w_{j+1} & > & \frac{\sum\limits_{i=1}^{j+1}\chi_i - rj}{\chi} - \frac{\sum\limits_{i=1}^j\chi_i - rj}{\chi} 
       =  \frac{\chi_{j+1}}{\chi} 
       >  0. \nonumber
\end{eqnarray*}
After choosing $w_1,\dots,w_{n-1}$ in this way, we set $w_n$ to be $1-\sum\limits_{i=1}^{n-1}w_i$. This gives us the required polarization and completes the proof. \\
$\underline{Case\; 3: \; \chi > 0}.$  In this case, we assume that $\chi = \sum\limits_{i=1}^n \chi_i - (n-1)r$ and that for each $I \subset \{1,\dots,n\}$,
\begin{align}\label{basic assumption for chi > 0}
   \sum\limits_{i \in I}\chi_i > (|I|-1)r. 
\end{align}
Here $|I|$ denotes the number of elements in $I$ (In particular, this also means, each $\chi_i > 0$). We prove this case 
also by applying induction on $n$. Again as in $Case\; 2$,  multiplying \ref{another form of inequalities for w1} throughout by $(\frac{1}{\chi})$ and using the fact that $\chi > 0$, we get 
 \begin{equation} \label{first inequality for w1 when chi > 0}
    \frac{\chi_1 - r}{\chi} < w_1 < \frac{\chi_1}{\chi}.   
 \end{equation}
Since each $\chi_i > 0$, we have $\frac{\chi_1}{\chi} > 0$. We now claim $\frac{\chi_1 - r}{\chi} < 1$. Notice that,  
\begin{eqnarray*}
\frac{\chi_1 - r}{\chi} < 1 &\Leftrightarrow&  \chi_1 - r < \chi = \sum\limits_{i=1}^n \chi_i - (n-1)r \\
&\Leftrightarrow& (n-2)r < \sum\limits_{i=2}^n \chi_i,
\end{eqnarray*}
which is true by our assumption (see inequalities \eqref{basic assumption for chi > 0}). Now since $\frac{\chi_1}{\chi} >0$ and $\frac{\chi_1-r}{\chi} < 1$, we can choose $w_1$ from $(0,1) \cap (0,\frac{\chi_1+\chi_2-r}{\chi}) \cap \mathbb{Q}$. \\
Now assume $j \geq 1$ to be a positive integer strictly less than $n-1$. Suppose that the $w_i$ for $1\leq i \leq j$ are chosen from $(0,1) \cap \mathbb{Q}$ such that $\sum\limits_{i=1}^jw_i \in (0,1) \cap (0, \frac{\sum\limits_{i=1}^{j+1} \chi_i-jr}{\chi}) \cap \mathbb{Q}$, satisfies \vspace{-0.07in}
\[
\frac{\sum\limits _{i=1}^j\chi_i - rj}{\chi} < \sum\limits_{i=1}^{j}w_i < \frac{\sum\limits_{i=1}^j\chi_i - r(j-1)}{\chi}.
\]
Equivalently, we have \vspace{-0.05in}
\begin{equation*}
  (\sum\limits_{i=1}^jw_i)\chi - \sum\limits_{i=1}^{j-1}\chi_i + r(j-1) <
  \chi_j < (\sum_{i=1}^jw_i)\chi - \sum\limits_{i=1}^{j-1}\chi_i + rj.
 \end{equation*}
  We want to prove that there exists $w_{j+1} \in (0,1) \cap \mathbb{Q}$ such that $\sum\limits_{i=1}^{j+1}w_i \in (0,1) \cap \mathbb{Q}$ and satisfying the inequality \eqref{inequalities for w_(i+1)}. Again procceding as in the previous case, the inequality \eqref{inequalities for w_(i+1)} is equivalent to 
  \begin{equation}\label{second inequality for w_(i+1) when chi>0}
   \frac{\sum\limits_{i=1}^{j+1}\chi_i - r(j+1)}{\chi} < \sum\limits_{i=1}^{j+1}w_i <  \frac{\sum\limits_{i=1}^{j+1}\chi_i - rj}{\chi}.   
 \end{equation}
 By inequality \eqref{basic assumption for chi > 0}, $\frac{\sum\limits_{i=1}^{j+1}\chi_i - rj}{\chi} > 0$. We claim that $\frac{\sum\limits_{i=1}^{j+1}\chi_i - r(j+1)}{\chi} < 1$. Suppose this claim is true. Then, since $\sum\limits_{i=1}^j w_i < \frac{\sum\limits_{i=1}^{j+1}\chi_i - rj}{\chi}$ by induction hypothesis, we can choose $w_{j+1} \in (0,1) \cap \mathbb{Q}$ such that $\sum\limits_{i=1}^{j+1} w_i \in (0,1) \cap \mathbb{Q}$. Once we have chosen $w_1,\dots,w_{n-1}$ in this way, we can define $w_n$ to be $1-\sum\limits_{i=1}^{n-1}w_i$. This gives us the required polarization and completes the proof. So, all we need to do is to prove the claim $\frac{\sum\limits_{i=1}^{j+1}\chi_i - r(j+1)}{\chi} < 1$. This can be seen as follows:
\begin{eqnarray*}
\frac{\sum\limits_{i=1}^{j+1}\chi_i - r(j+1)}{\chi} < 1  &\Leftrightarrow& \sum\limits_{i=1}^{j+1}\chi_i - r(j+1) < \chi = \sum\limits_{i=1}^n \chi_i - (n-1)r \\
&\Leftrightarrow& (n-(j+2))r < \sum\limits_{i=j+2}^n \chi_i,
\end{eqnarray*}
which is true by inequality \eqref{basic assumption for chi > 0}.
\end{proof}

\begin{remark}
Let $E$ be as in hypothesis of Theorem \ref{Section-5}. If $(\chi_1,\dots,\chi_n)\in \mathcal{W}_{r,n}$, then there exists a 
polarization $w$ such that $E$ is $w$-stable.
\end{remark}

\section*{Acknowledgments}
We thank Arijit Dey for suggesting this problem as well as for helpful discussions. We also thank D. S. Nagaraj for his fruitful comments on this manuscript. We would like to express our gratitude to the anonymous referee for valuable suggestions and  helpful remarks. The third named author would like to thank National Board for Higher Mathematics (NBHM), Department of Atomic Energy, Government of India, for financial support through Postdoctoral Fellowship. He also thanks The Institute of Mathematical Sciences, Chennai, where the majority of the work was carried out.


\begin{thebibliography}{99}
\bibitem{Barik-Dey-Suhas-2018}P. Barik, A. Dey, Suhas B N, \textit{On the rationality of Nagaraj-Seshadri moduli space}, Bull. Sci. Math. {\bf 140} (2016), o. 8, 990-1002.
\bibitem{Bhosle-92} Usha N. Bhosle, \textit{Generalised parabolic bundles and applications to torsionfree sheaves on nodal curves}, Ark. Mat. 30 (1) (1992) 187-215.
\bibitem{Bhosle-Biswas-2014} Usha N. Bhosle, Indranil Biswas, \textit{Brauer group and birational type of moduli spaces of torsion-free sheaves on a nodal curve}, Comm. Algebra 42 (4) (2014) 1769-1784.
\bibitem{Biswas-Paz-Newstead} I. Biswas, L. Brambila-Paz, P. E. Newstead, \textit{Stability of projective Poincaré and Picard bundles}, Bull. Lond. Math. Soc. 41 (2009), no. 3, 458–472.
\bibitem{Brivio-Favale-VB-2020} Sonia Brivio, Filippo F. Favale, \textit{On vector bundles over reducible curves with a node}, Adv. Geom. 21 (2021), no. 3, 299-312.
\bibitem{Dey-Suhas-2018} A. Dey, Suhas, B N,  \textit{Rationality  of  moduli  space  of  torsion-free  sheaves  over  reducible  curve}, J. Geom. Phys. {\bf 128} (2018), 87-98.


\bibitem{Hart-1977} R. Hartshorne, \emph{Algebraic Geometry} Springer-Verlag, New York-Heidelberg, 1977, Graduate Texts in Mathematics, No. 52.



\bibitem{King-Schofield-1999} Alastair King, Aidan Schofield, \textit{Rationality of moduli of vector bundles on curves}. Indag. Math. 10 (4) (1999)519 535.

\bibitem{Mumford-GIT-1965} D. Mumford, \textit{ Geometric invariant theory}, Ergeb. Math. Grenzgeb. Neue Folge, vol. 34, Springer-Verlag, Berlin - New York, 1965.




\bibitem{Narasimhan-Ramanan-1969} M.S. Narasimhan, S. Ramanan, \textit{Moduli of vector bundles on a compact Riemann surface}, Ann. Math. {\bf 89} (1969) 14-51.
\bibitem{Narasimhan-Ramanan-1978} M.S. Narasimhan, S. Ramanan, \textit{Geometry of Hecke cycles. I}, C. P. Ramanujam-a tribute, Tata Inst. Fund. Res. Stud. Math. 8 (1978) 291-345.
\bibitem{Newstead-1975} P. E. Newstead, \textit{Rationality of moduli spaces of stable bundles}, Math. Ann. {\bf 215} (1975) 251-268. 
\bibitem{Newstead-1980} P. E. Newstead, Correction to: \textit{Rationality of moduli spaces of stable bundles}. Math. Ann. 249: (1980), 281-282.
\bibitem{Newstead-2011} P.E. Newstead, \textit{Introduction to Moduli Problems and Orbit Spaces}, vol. 17, Tata Institute of Fundamental Research Publications, 2011.

\bibitem{Seshadri-1982} C. S. Seshadri, \textit{Fibrés vectoriels sur les courbes algébriques}, Astérisque 96, Société Math. France, Paris, 1982.
\bibitem{Suhas-susobhan-amit-2021} Suhas B N, Susobhan Mazumdar, Amit Kumar Singh, \textit{On the stability of kernel bundles over chain-like curves}, J. Geom. Phys. {\bf164} (2021).
 \bibitem{Bigas91} M. Teixidor i Bigas, \emph{Moduli spaces of (semi)stable vector bundles on tree-like curves}, Math. Ann. {\bf 209}, (1991), 341-348.

\bibitem{Tjurin-1966} A.N. Tjurin, \textit{Classification of n-dimensional vector bundles over an algebraic curve of arbitrary genus},  Izv. Akad. Nauk SSSR Ser. Mat. 30 (1966) 1353-1366.
\end{thebibliography}
 \end{document}